\documentclass[a4paper,10pt,leqno]{amsart}
\usepackage[total={6in,9in},
top=1in, left=1in, right=1in, bottom=1in]{geometry}
\usepackage{amsmath}
\usepackage{amsfonts}
\usepackage{amssymb}
\usepackage[english]{babel}
\usepackage{dsfont}
\usepackage{mathrsfs}
\usepackage{graphicx}
\usepackage{setspace}
\usepackage{fancyhdr}
\usepackage{amsthm}
\usepackage{empheq}
\usepackage{cases}
\usepackage[all]{xy}
\usepackage{stmaryrd}
\usepackage{color}
\usepackage{tocvsec2}

\newcommand{\erre}{\mathbb{R}}
\newcommand{\erren}{\mathbb{R}^n}

\newcommand{\ricc}{\operatorname{Ric}}

\newcommand{\Hess}{\operatorname{Hess}}

\newcommand{\sff}{\mathrm{II}}

\newcommand{\pa}[1]{{\left(#1\right)}}                  % tra tonde
\newcommand{\abs}[1]{{\left|#1\right|}}                 % valore assoluto
\newcommand{\metric}{g}                          % metrica
\renewcommand{\tilde}[1]{\widetilde{#1}}

\newcommand{\tx}[1]{\mbox{\;{#1}\;}}
\newcommand{\p}{\partial}

\newtheorem{theorem}{\textbf{Theorem}}[section]
\newtheorem{lemma}[theorem]{\textbf{Lemma}}
\newtheorem{proposition}[theorem]{\textbf{Proposition}}
\newtheorem{cor}[theorem]{\textbf{Corollary}}

\newtheorem{rem}[theorem]{\textbf{Remark}}

\newtheorem{exe}[theorem]{\textbf{Example}}
\numberwithin{equation}{section}
\title{Reaction-diffusion problems \\ on time-dependent Riemannian manifolds:\\  stability of periodic solutions}

\date{\today}



\begin{document}
\maketitle

\begin{center}
\textsc{\textmd{C. Bandle\footnote{University of Basel, Switzerland. Email: c.bandle@gmx.ch.}, D. D. Monticelli\footnote{Politecnico di Milano, Italy. Email:
dario.monticelli@polimi.it.} and F. Punzo\footnote{Politecnico di Milano, Italy. Email: fabio.punzo@polimi.it. \\  D.D.
Monticelli and F. Punzo are supported by GNAMPA projects 2017 of INdAM. D. D. Monticelli is partially supported  by the PRIN-2015KB9WPT Grant:
``Variational methods, with applications to problems in mathematical physics and geometry''}}}
\end{center}

\begin{abstract}
We investigate the stability of time-periodic solutions of semilinear parabolic problems with Neumann boundary conditions. Such problems are posed on compact submanifolds evolving periodically in time. The discussion is based on the principal eigenvalue of periodic parabolic operators. The study is motivated by biological models on the effect of growth and curvature on patterns formation. The Ricci curvature plays an important role.
\end{abstract}

\maketitle

%\tableofcontents

\section{Introduction}
Stable stationary nonconstant solutions of reaction-diffusion equations play an important role in the study of patterns,  which are of great interest in mathematical biology. Moreover, problems where the data depend periodically on time arise naturally in population ecology. An example is the  {\sl T-periodic Fisher  model}
\begin{equation*}
\begin{cases}
\frac{\p u}{\p t} -\Delta u= m(x,t)h(u) &\tx{in } \Omega\times (0,\infty),\\
\frac{\p u}{\p \nu} =0 &\tx{on } \p \Omega\times (0,\infty).
\end{cases}
\end{equation*}
Here $\Omega \subset \erren$ is a bounded Lipschitz domain, $\nu$ is its outer normal, $m(x,t)=m(x,t+T)$ may change sign and $h(u)= u(1-u)[\alpha (1-u)+(1-\alpha)u]$ for some $0<\alpha <1$. It has been shown that this problem possesses $T$-periodic solutions, see \cite{Hess, HW}.
As in the case of stationary solutions, the question of their stability arises. It is related to the principal eigenvalue of the corresponding linearized problem. The principal eigenvalue corresponding to a time-periodic eigenfunction has been studied in detail by Hess and coworkers. Most results are collected in the Lecture Note \cite{Hess}. For nonlinear source terms $f(t,u,\nabla u)$ which do not depend on $x$, P. Hess
\cite{He} has shown that in a convex domain $\Omega\subset\erre^n$ all $T$-periodic solutions are unstable. Our main goal is to extend the results of Hess to problems on Riemannian manifolds.

In this paper we consider reaction-diffusion equations on submanifolds  in a fixed Riemannian manifold, evolving periodically in time. In this case the metrics of the submanifolds  depend on time. The motivation comes from biology, where models have been developed for substances which occupy domains depending on time. We refer in particular to the paper by Maini et al. \cite{EFG}.

Similar problems where the equation, the metric and the domain are time independent  have been considered in \cite{Ji} and \cite{BPT}.
In all these papers the Ricci curvature and the convexity of the underlying domain is crucial for the stability of solutions.

Several authors have studied the heat equation on manifolds with time-dependent metrics, see for instance  \cite{Chow} and \cite{Gue}
in connection with the Ricci flow. Moreover in \cite{EFG} semilinear parabolic equations on evolving surfaces of $\mathbb R^3$ have been considered, in particular the Turing instability has been discussed.

In our paper the most illuminating examples are evolving surfaces of revolution in $\mathbb{R}^3$, including dilations of spheres and cones. We give criteria for the instability of time periodic solutions, depending on various curvatures. In the general case the Ricci curvature and the convexity of the underlying submanifolds come into play.

Our paper is organized as follows. In Section 2, for the reader's convenience, we collect some general results on time-periodic parabolic problems and a criterion for stability and instability. In Section 3 we treat problems on surfaces with rotational symmetry. Section 4 is devoted to general Riemannian manifolds. For the reader's convenience we give a short introduction to the main geometric concepts and tools, such as the Ricci curvature and the Bochner-Weitzenb\"ock formula which will be needed later in the proof of instability. Moreover we also relate our general problem to a biological model by Maini et al. \cite{EFG}.

%%%%%%%%%%%%%%%%%%%%%%%%%%%%%%%%%%%%%%%%%%%%
%%%%%%%%%%%%%%%%%%%%%%%%%%%%%%%%%%%%%%%%%%%%
\section{Known results for T-periodic solutions}
In this section we collect some results on periodic solutions of parabolic problems. Most of the material is taken from \cite{Hess} and the references cited therein.

Let $Q_T\subset \erre ^{n+1}$ be a bounded Lipschitz domain of the form
$$Q_T=\bigcup_{0< t\leq T} \Omega_t \times \{t\},$$ where $\Omega_t \subset \mathbb{R}^m$ for each $t\in[0,T]$. We require that $\Omega_t= \Omega_{t+T}$.  With $(x,t)$ where $x\in \Omega_t$ and $0\leq t\leq T$ we denote a point in $Q_T$.
 Let  $\nu_t=(\nu_1(t),\cdots,\nu_n(t))$ be the outer normal of $\Omega_t$.  Denote by $\Gamma_T$ the parabolic boundary $$\Gamma_T:=  \bigcup_{0\leq t<T}\p \Omega_t \times \{t\}.$$

 In $Q_T$ we consider the elliptic operator
 $$
 A(x,t)=- \sum_{i,j=1}^n a_{ij}(x,t) \frac{\p^2}{\p x_i\p x_j} +\sum_{i=1}^n a_i(x,t) \frac{\p}{\p x_i} +a_0(x,t),
 $$
where  all coefficients are supposed to be H\"older continuous with respect to the parabolic metric  $$\left(|x-y|^2+ |s-t|\right)^{1/2}.$$ More precisely $a_{ij},a_i,a_0 \in C^{\mu,\mu/2}(\overline{Q}_T)$. The first exponent refers to the $x$- and the second to the $t$-derivatives. In addition we assume that they are $T$-periodic, {\sl  i.e.} $a_{ij}(x,t)=a_{ij}(x,t+T)$, $a_i(x,t)=a_i(x,t+T)$ for $i,j=0,\cdots,n$.
%%%%%%%%%%%%%%%%%%%%%%%%%
 \subsection{The eigenvalue problem}The eigenvalue problem which will be crucial  for our arguments is
\begin{equation}\label{eq:eigenvalue}
\begin{cases}
\displaystyle\frac{\p \phi}{\p t} +A(x,t) \phi=\lambda \phi& \tx{in } Q_T,\\
 \displaystyle\sum_{i,j=1}^n a_{ij}(x,t)\frac{\p \phi}{\p x_i} \nu_j(t) =0& \tx{on } \Gamma_T,\\
\phi(x,t+T)=\phi(x,t).
\end{cases}
\end{equation}

The following result has been proved in \cite{BH} and \cite{Hess} for the case of a cylindrical domain $Q_T=\Omega_0\times(0,T)$.  If we assume that $\Omega_t$ is isomorphic to $\Omega_0$, we can make a change of coordinates such that the transformed equation is defined in a cylindrical domain.

%%%%%%%%%%
\begin{lemma}\label{l:eigenvalue}
Assume that $\sum_{i=1}^n a_{ij}(x,t) \nu_j(t)$ is not tangential to $\Omega_t$. Then
\begin{enumerate}
\item[i)] $\lambda$ does not depend on $t$.
\item[ii)]  The spectrum of \eqref{eq:eigenvalue} is discrete.
\item[iii)] There exists a principal eigenvalue $\lambda_1$ which is real, simple and has an eigenfunction of constant sign.
\item[iv)] Re$\{\lambda\} >\lambda_1$ for all other eigenvalues $\lambda$.
\item[v)] If $a_0\geq 0$, $a_0\neq 0$, then $\lambda_1 >0$.
\end{enumerate}
\end{lemma}
%%%%%%%%%%%%%%%%%
In contrast to the elliptic case, $\lambda_1$ has no variational characterization. Estimates can be found in \cite{Hess}. Our arguments will be based on the following lemma which was also used in \cite{DH, He}. We provide here a slightly different proof.
%%%%%%%%%%%%%%%%%%%%%LLLLLLLLLL
\begin{lemma}\label{l:estimates}
 Assume that there exists a positive, T-periodic function $w\in L^2(0,T;W^{1,2}(\Omega_t))\cap C^1(\overline{Q_T})$, $w\geq 0$, $w\not\equiv 0$ such that
 \begin{equation*}
 \begin{cases}
 \displaystyle   \frac{\p w}{\p t}+A(x,t)w\leq 0& \tx{in } Q_T \tx{in the weak sense},\\
 \displaystyle\sum_{i,j=1}^n a_{ij}(x,t)\frac{\p w}{\p x_i} \nu_j (t)\leq 0&  \tx{on } \Gamma_T.
 \end{cases}
 \end{equation*}
Then $\lambda_1\leq 0$.

Similarly if $w$ satisfies
 \begin{equation*}
 \begin{cases}
 \displaystyle   \frac{\p w}{\p t}+A(x,t)w\geq 0& \tx{in } Q_T \tx{in the weak sense},\\
 \displaystyle\sum_{i,j=1}^n a_{ij}(x,t)\frac{\p w}{\p x_i} \nu_j (t)\geq 0&  \tx{on } \Gamma_T.
 \end{cases}
 \end{equation*}
Then $\lambda_1\geq 0$.
\end{lemma}
%%%%%%%%%%%%%%%%%%%%%%%%%%%llllllllllllllll
\begin{proof}
  Let $\phi>0$ be the eigenfunction corresponding to the principal eigenvalue of \eqref{eq:eigenvalue}. Since it is of constant sign we can write
$w=v\phi$. Then $v $ satisfies in the weak sense
\begin{align}
&\frac{\p v}{\p t} - \sum_{i,j=1}^n a_{ij}(x,t) \frac{\p^2 v}{\p x_i\p x_j} +\sum_{i=1}^n a_i(x,t) \frac{\p v}{\p x_i}  -2\phi^{-1}\sum_{i,j=1}^n a_{ij} \frac{ \p v}{\p x_i}\frac{ \p \phi}{\p x_j}+\lambda_1 v \leq 0\qquad  \tx{in } Q_T, \label{eq:subsolution1}\\
&\sum_{i,j=1}^n a_{ij}(x,t)\frac{\p v}{\p x_i} \nu_j (t)\leq 0 \qquad\tx{on } \Gamma_T, \label{eq:subsolution2}\\
&v(x,t)=v(x,t+T)\geq 0 \label{eq:subsolution3}.
\end{align}
Assume that $\lambda_1> 0$. By the maximum principle for weakly subparabolic functions \cite{Fr}, $v$ assumes its maximum on $\Gamma_T \cup \Omega_0$ unless $v\equiv v_{\max}$ in $Q_T$. This is impossible by \eqref{eq:subsolution1}.  Since $\sum_{i,j=1}^n a_{ij}(x,t)\frac{\p v}{\p x_i} \nu_j (t)\leq 0$ on $\p \Omega_t\times \{t\}$  the strong maximum principle implies that $v$ cannot take its maximum on $\Gamma_T$, for $0\leq t\leq T$. If it takes its maximum on $\Omega_0$, in view of the periodicity it assumes its maximum also on $\Omega_T\times \{T\}$ and it is therefore constant.
%which is an inner point of $Q_T$.  In this case $v\equiv v_{\max}$  in $Q_T$. By assumption $v$ is positive.
By \eqref{eq:subsolution1} this is impossible, and therefore $\lambda_1\leq 0$. The proof of the other assertion is similar.
\end{proof}

%%%%%%%%%%%%%%%%%%%%%%%%%%%%%%%%%%%%%%%%%%ssssssssss
\subsection{Semilinear parabolic problems}
Consider the quasilinear periodic-parabolic boundary value problem
\begin{equation}\label{eq:semilinear}
\begin{cases}
\displaystyle\frac{\p u}{\p t} +A(x,t)u= f(x,t,u,\nabla u) &\tx{in } Q_T,\\
 \displaystyle\sum_{i,j=1}^n a_{ij}(x,t)\frac{\p u}{\p x_i} \nu_j (t)=0& \tx{on } \Gamma_T,\\
u(\cdot,t)=u(\cdot, t+T).
\end{cases}
\end{equation}
Existence results for large classes of nonlinearities have been derived in \cite{ A, DH}.
We are interested in the \textit{linearized stability} of the solutions of \eqref{eq:semilinear}. We assume that $f$ is of class $C^1$ and $T-$periodic for $t\in \mathbb R$. The corresponding linearized problem is
\begin{equation}\label{eq:linearized}
\begin{cases}
\displaystyle\frac{\p \phi}{\p t} +A(x,t)\phi-f_u(x,t,u,\nabla u)\phi-\sum_{i=1}^n f_{u_{x_i}}(x,t,u,\nabla u) \phi_{ x_i}=\lambda_1 \phi \qquad\tx{in } Q_T,\\ \displaystyle\sum_{i,j=1}^n a_{ij}(x,t)\phi_{ x_i} \nu_j (t)=  0\qquad \tx{on } \Gamma_T,\\
\phi(x,t+T)=\phi(x,t)\qquad x\in\Omega_t.
\end{cases}
\end{equation}
A similar criterion as for the stationary solutions holds for the $T$-periodic solutions of \eqref{eq:semilinear}. It goes back to D. Henry  \cite{DaH}.
%%%%%%%%%%%%%%%%%%%%%%%%%%%%%%%%%%%%%%%%LLLLLLLLLLLLLLLLLLL
%%%%%%%%%%%%%%%%%%%%%%%%%%%%%%%%%%%%%%
\begin{lemma}\label{l:instability} (i) If $\lambda_1 >0$, then $u$ is stable.
\smallskip

(ii) If $\lambda_1<0$, then $u$ is unstable.
\end{lemma}
%%%%%%%%%%%%%%%%%%%%%%%%%%%%%%lllllllll
The proof is found in \cite[pages 247-250]{DaH}.
%%%%%%%%%%%%%%%%%%%%%%%%%%%%%%%%%%%%%%%%%%%%%%%%%%%%%%%%%%%%%%%%%%%%%%%%%%
%%%%%%%%%%%%%%%%%%%%%%%%%%%%%%%%%%SSSSSSSSSSS
\section{Surfaces of revolution in $\mathbb R^3$ evolving in time}\label{surfaces}
\subsection{Instability results} In this section we consider compact surfaces of revolutions which are the locus of points in $\mathbb{R}^3$ generated by rotating the regular plane curve $r \to (\psi(r,t),
\chi(r,t))$  around the $z$-axis. We assume that for each $r\in I :=(a(t),b(t)),$ the functions $t\mapsto \psi(r, t)$, $t\mapsto \chi(r, t)$, $t\mapsto a(t)$, and $t\mapsto b(t)$, are $T-$periodic and $\psi(r, t)>0$ for every $t\in\erre$ and $r\in I$. This leads to a family of time-dependent surfaces of revolutions $\Omega_t$, parametrized by
 \begin{equation}\label{e2fn}
\left\{
\begin{array}{ll}
 \, x = \psi(r, t)\cos \theta
\\& \\
\textrm{ }y\,= \psi(r, t)\sin \theta,  \\&\\
 \textrm{ }z\, = \chi(r, t).
\end{array}
\right.\qquad (r,\theta, t)\in [a(t),b(t)]\times[0,2\pi)\times \mathbb {R}^+
\end{equation}
The metric depends on $t$ and is given by
\begin{equation*}%\label{metric}
ds^2= q^2(r,t)\,d r^2 + \psi^2(r, t) d \theta^2, \quad\textrm{ where } q=\sqrt{\psi_r^2 +\chi_r^2}
\end{equation*}
for $(r,\theta)\in I\times[0,2\pi)$. The  Laplace-Beltrami operator  is expressed as
\begin{align}\label{e3fbis}
\Delta_t u \,=\, \frac{1}{\psi q}\frac{\partial}{\partial r}\left(\frac{\psi}{q}\,u_r\right)
+\,
\frac 1{\psi^2}\frac{\partial^2 u}{\partial\theta^2},
\end{align}
and the Ricci (Gaussian) curvature of $\Omega_t$
is
\begin{equation}\label{e7fbis}
R(r, t)=\frac{-\psi_{rr}\chi_r^2 +\psi_r\chi_r\chi_{rr}}{\psi q^4}\,.
\end{equation}
The boundary
$\partial \Omega_t$ consists of the time-dependent  geodesic circles
\begin{eqnarray*}
C_{a,t}&:=& \{(\psi(a(t), t) \cos\theta, \psi(a(t), t) \sin\theta,\chi(a(t), t))\,
|\,\theta\in  [0,2\pi),\, t \in \mathbb R\}\,, \\
C_{b,t}&:=& \{(\psi(b(t), t) \cos\theta, \psi(b(t), t) \sin\theta,\chi(b(t), t)) \,
|\,\theta\in  [0,2\pi),\, t \in \mathbb R\}\,.
\end{eqnarray*}
For the sake of simplicity we shall assume that $\chi_r(a(t), t)>0$, $\chi_r(b(t), t)>0$ for every $t\in \mathbb {R}^+$.
\medskip

We can reduce our problem on cylindrical domains, if we replace
the variable $r$ by $\rho=\frac{r-a}{b-a}$. The metric of $\tilde\Omega_t$ then becomes
$$
\tilde ds^2= ( \tilde \psi_\rho^2 + \tilde \chi_\rho^2) d\rho^2 + \tilde \psi^2d\theta^2, \tx{where} \tilde \psi(\rho,t)=\psi(r(\rho),t),\,\tilde \chi(\rho,t)=\chi(r(\rho),t).
$$
The Ricci curvature and the Laplace-Beltrami operator are changed accordingly.
In the new variables
\begin{align*}
\tilde \Omega_t&:= \{(\tilde\psi(\rho, t) \cos\theta, \tilde\psi(\rho, t) \sin\theta,\tilde\chi(\rho, t)) \,
|\,(\rho,\theta, t)\in  (0,1)\times [0,2\pi)\times \mathbb R\},\\
\tilde C_{0,t}&:= \{(\psi(0, t) \cos\theta, \psi(0, t) \sin\theta,\chi(0, t))\,
|\,\theta\in  [0,2\pi),\, t \in \mathbb R\}\,, \\
\tilde C_{1,t}&:= \{(\psi(1, t) \cos\theta, \psi(1, t) \sin\theta,\chi(1, t)) \,
|\,\theta\in  [0,2\pi),\, t \in \mathbb R\}\,.
\end{align*}
As an illustration consider the following
\begin{exe} Suppose that a substance occupies a spherical cap  $\Omega_t$ on the sphere $\{|x|=R(t):x\in \mathbb{R}^3\}$. Let $r(t)\in (0,b(t))$ be the distance to the North Pole. Then $\Omega_t$ is described by
$$
(x,y,z)= R(t)\left(\sin\frac{r}{R(t)}\cos \theta, \sin\frac{r}{R(t)}\sin \theta, \cos\frac{r}{R(t)}\right),\quad (r,\theta,t)\in [0,b(t))\times [0,2\pi)\times (0,T).
$$
In this case we have
$$\psi(r,t)=R(t)\sin\frac{r}{R(t)}, \quad \chi(r,t)=R(t)\cos\frac{r}{R(t)} \tx{and} ds^2= dr^2+R^2(t)\sin^2\frac{r}{R(t)}d\theta^2,  \:b(t)\in (0, \pi R(t)).
$$
The boundary of $\Omega_t$ is
$$
C_{b,t}= R(t)\left(\sin\frac{b(t)}{R(t)}\cos \theta, \sin\frac{b(t)}{R(t)}\sin \theta, \cos\frac{b(t)}{R(t)}\right),\quad \theta\in (0,2\pi),\: t\in (0,T).
$$
\end{exe}
\medskip

We now discuss the reaction-diffusion problem in  $\Omega_t$ with boundaries $C_{0,t}$ and $C_{1,t}$,
\begin{equation}\label{e1fn}
\left\{
\begin{array}{ll}
\displaystyle {\frac{\partial u}{\partial t}= \Delta_{t} u +f(r,\theta,t,u,u_r,u_\theta) }
 & \textrm{in} \;\; Q_T\\& \\
\displaystyle{\frac{\partial u}{\partial r}}  =0  &
\textrm{on} \;\; \Gamma_T\\& \\
u(\cdot, t) = u(\cdot, t +T) & \textrm{in} \;\;  \Omega_t,\, \forall\,\, t\in \mathbb R\,.
\end{array}
\right.
\end{equation}
Lemma \ref{l:estimates}  obviously applies to this problem.

We start with a simple well-known observation.
%%%%%%%%LLLLLLLL
\begin{lemma}\label{lemma33} Suppose that $f=f(r,t,u,u_r,u_\theta)$ does not depend explicitly on $\theta$. Then any non-radial solution of \eqref{e1fn} is unstable.
\end{lemma}
%%%%%%%%%%%%lllll
\begin{proof}
We differentiate the differential equation  \eqref{e1fn} with respect to $\theta$ and observe that $u_\theta$ is an eigenfunction corresponding to the eigenvalue $\lambda=0$ of the linearized problem
\begin{equation*}
\begin{cases}
\phi_t=\Delta_t\phi+f_{u}(r,t,u,u_r,u_\theta)\phi+f_{u_r}(r,t,u,u_r,u_\theta)\phi_r+f_{u_\theta}(r,t,u,u_r,u_\theta)\phi_\theta+\lambda \phi \tx{in}Q_T\\
\phi_r =0 \tx{on}\Gamma_T\\
\phi(\cdot, t) = \phi(\cdot, t +T)  \tx{in} \Omega_t,\, \forall\, t\in \mathbb R\,.
\end{cases}
\end{equation*}
Since $u_\theta$ changes sign, $0$ cannot be the principal eigenvalue and therefore $\lambda_1<0$. This together with  Lemma \ref{l:eigenvalue} establishes the assertion.

 \end{proof}

Consider now the solutions which are independent of $\theta$.
Hess used in his paper \cite{He}, Casten and Holland's trick to prove the instability of  solutions of \eqref{eq:semilinear} in domains $\Omega\subset\mathbb{R}^n$, which depend on the space variable. % For time-dependent solutions this method seemed to apply only for Neumann boundary conditions.
This idea has been used also in \cite{BPT} for stationary radial solutions on surfaces of revolutions. The next result is an extension of \cite[Thm.5.2]{BPT}.
%%%%%%%%%%%%%%%%%%%%%%%%%%%%%%%%%TTTTTTTTT
\begin{theorem}\label{thmB}
Let $u(r,t)$ be a radial solution of \eqref{e1fn} such that $u_r\neq 0$. Assume that $f=f(t,u,u_r,u_\theta)$ is independent of $r$. If
\begin{equation}\label{18}
\frac 1 q \left (\frac{\psi_r}{q\psi}\right)_r -\frac{q_t}{q}+f_{u_r}(t,u,u_r,0)\frac{q_r}{q}\leq 0,
\end{equation}
then $u$ is unstable.
\end{theorem}
%%%%%%%%%%%%%%%%%%%%%%%%%%%ttttttttttttt
\begin{proof}
We consider the function $v=\frac{|u_r|}{q}$, which, in view of our assumptions, is nontrivial  and we shall show that it satisfies the assumptions of Lemma \ref{l:estimates} in a weak sense.

In the region $Q^+_T$ where $u_r$ is positive, we have
$$
u_r=vq,\tx{and} u_{rr}= v_rq+vq_r \tx{and} u_{tr}= v_t q+ vq_t.
$$
With this notation \eqref{e1fn}  is  expressed  as
%\begin{align*}
%u_t= \frac{v_{rr}}{q} +v_r\frac{1}{\psi}\left(\frac{\psi}{q}\right)_r +v\left[\left(\frac{q_r}{q^2}\right)_r+ \psi^{-1}\left(\frac{\psi}{q}\right)_r\right] +f, \tx{in} \Omega_t^+.
%\end{align*}
\begin{align*}
u_t \,=\, \frac{v_r}{q} + v \frac{\psi_r}{q\psi}+ f \quad \textrm{in}\;\; Q^+_T\,.
\end{align*}
If we differentiate this equation with respect to $r$, we get
\begin{align}\label{eq:v}
v_t= \Delta_t v +v\left[\frac{1}{q}\left(\frac{\psi_r}{q\psi}\right)_r -\frac{q_t}{q}+f_{u_r}\frac{q_r}{q}\right]  + f_u v+f_{u_r}v_r\,.
\end{align}
The same equation holds in the region $Q^{-}_T$
where $u_r$ is negative. Under our assumptions we have
\begin{align}\label{i:v}
v_t- \Delta_t v \leq    f_u v+f_{u_r}v_r \quad \tx{in}\;\;  Q^+_T\cup Q^-_T.
\end{align}
By Kato's inequality \eqref{i:v} holds in $Q_T$ in the weak sense. On the boundary $v$ satisfies, in view of the Neumann boundary conditions, $v=0$. Hence $\frac{\partial v}{\partial\nu}\leq 0$. By Lemma \ref{l:estimates} we deduce that $\lambda_1\leq0$.
The case $\lambda_1=0$ can be excluded by a contradiction argument which holds in a more general case and which is proved in Theorem \ref{thmA}. Then $\lambda_1<0$. The instability is now a consequence of Lemma \ref{l:instability}.
\end{proof}

\begin{rem}\label{rem33}
The geodesic curvature of the parallel circles $r=const$ for each $t\in [0, T]$ is
\[k_g(r, t)=\frac{\psi_r}{\psi \sqrt{\chi_r^2+\psi_r^2}}\,.\]
A simple computation yields
\begin{equation}\label{e200}
\frac 1 q \left (\frac{\psi_r}{q\psi}\right)_r = - R(r,t) - k_g^2(r, t)\,.
\end{equation}
Hence condition \eqref{18} can be written as
\[-R - k_g^2 \leq  \frac{q_t}{q^2}- f_{u_r}(t,u,u_r,0)\frac{q_r}{q^2}\,.\]
In the special case where $q=1$ and $f=f(u)$, this condition becomes
\[- R - k_g^2 =  \left(\frac{\psi_r}{\psi}\right)_r\leq 0,\]
which is in accordance with the results in \cite{BPT}.
\end{rem}

\begin{exe} Consider the  dilations of the unit sphere of the form $$(x,y,z)=(\rho_1(t)\sin r\cos\theta, \rho_1(t)\sin r,\sin \theta, \rho_2(t)\cos r)$$ where $\rho_i(0)=\rho_i(T)=1$ for $i=1,2$, $r\in (0,1)$, $\theta \in (0,\pi)$. In this case
$$
q^2= \rho_1^2(t)\cos^2r +\rho_2^2(t)\sin^2 r,\quad \psi^2(r,t)=\rho^2_1(t)\sin^2 r,
$$
and
$$
\frac {1}{ q }\left (\frac{\psi_r}{q\psi}\right)_r -\frac{q_t}{q}=-\frac{1}{q^2\sin^2r}-\frac{(\rho^2_2-\rho^2_1)\cos^2r}{q^4}-\frac{\dot\rho_1\rho_1\cos^2r+\dot\rho_2\rho_2\sin^2r}{q^2}.
$$
In particular if $\rho_1=\rho_2=\rho$ then
$$
\frac {1}{ q }\left (\frac{\psi_r}{q\psi}\right)_r -\frac{q_t}{q}=-\frac{1+\rho\dot\rho\sin^2r}{\rho^2\sin^2r}.
$$

\end{exe}
\begin{exe} Next we consider the deformation of the cone of the type $$(x,y,z)=(\rho_1(t)r\cos \theta, \rho_1(t)r\sin\theta,\rho_2(t)r).$$
Here $q^2=\rho_1^2(t)+\rho_2^2(t) \tx{and} \psi^2=\rho^2_1(t)r^2,
$
and
$$
\frac {1}{ q }\left (\frac{\psi_r}{q\psi}\right)_r -\frac{q_t}{q}=-\frac{1}{(\rho^2_1+\rho^2_2)r^2}-\frac{\dot\rho_1\rho_1+\dot\rho_2\rho_2}{\rho_1^2+\rho_2^2}.
$$
\end{exe}
%%%%%%%%%%%%%%%%%%%%%%%%%%
%%%%%%%%%%%%%%%%%%%%%%%%%%%
\subsection{The construction of a stable solution} Next we show that if condition \eqref{18} is not verified, then there exist examples for which  T-periodic solutions $u$ with $u_r(r,t)\neq 0$  are stable. For this purpose we shall use a result of  \cite{BPT}.

Let $\Omega$ be a time-independent surface in $\mathbb{R}^3$ where $\psi^2_r +\chi^2_r=1$ and $r\in (0,1)$. Consider on $\Omega$ the problem
\begin{equation}\label{e30f}
\left\{
\begin{array}{ll}
- \Delta U  = F(U)
 & \textrm{in} \;\;  \Omega
\\& \\
\displaystyle{\frac{\partial U}{\partial \nu}}  =0  &
\textrm{in} \;\; \partial \Omega\,.
\end{array}
\right.
\end{equation}
Here $\Delta$ stands for the Laplace- Beltrami operator on $\Omega$.
By \cite[Theorem 4.1]{BPT},  there exists under the assumption $\left(\frac{\psi_r}{\psi}\right)_r>0$ a function $F\in C^1(\mathbb R)$ such that problem \eqref{e30f}  admits an
 asymptotically stable solution $U=U(r)$ satisfying $U_r(r)>0$ in $(0,1)$.

Let $\tilde\psi( r,t)=\zeta(t)\psi(r)$ and $\tilde\chi( r,t)=\zeta(t)\chi(r)$,   $\zeta(t+T)=\zeta(t)>0$ and consider the family of surfaces of revolution $\tilde \Omega_t$, defined as in Section 3.1. Note that $q=\zeta(t)$ and in view of  \eqref{e3fbis}
 $$
 \Delta_t U= \frac{1}{\zeta^2}\Delta U=-\frac{1}{\zeta^2}F(U).
 $$
Let $\phi>0 \in C^1(\mathbb R)$ be $T$-periodic and define
\[ u(r, t):= U(r) \phi(t), \qquad\mbox{ for } r\in [0, 1],\, t \in \mathbb R\,. \]
Then
\begin{align}\label{stable}
\partial_t u - \Delta_t u &= \frac{\phi'(t)}{\phi(t)} u - \frac{\phi(t)}{\zeta^2(t)} \Delta U = \frac{\phi'(t)}{\phi(t)} u + \frac{\phi(t)}{\zeta^2(t)} F\left(\frac{u}{\phi(t)} \right)=:f(t, u) \tx{in} \tilde \Omega_t,\\
\nonumber \frac{\partial u}{\partial \nu_t}&=0 \tx {on} \tilde C_{0,t}\cup \tilde C_{1,t}.
\end{align}
%%%%%%%%%%%%%%%%%%%%%%%%%%TTTTTTTTTTTTTTTTTTTTTT
\begin{theorem}\label{thm2f}
If
\begin{equation}\label{e34f}
\left (\frac{\psi_r}{\zeta\psi}\right)_r -\frac{\zeta_t}{\zeta}> 0,
\end{equation}
then $u$ is an asymptotically stable solution solution of problem \eqref{stable}
\end{theorem}
%%%%%%%%%%%%%%%%%%%%%%%%%ttttttttttttt
\begin{proof}  Observe that $v=u_r/\zeta$ satisfies
$$ \partial_t v - \Delta_t v = \left[\frac {1}{\zeta^2}\left(\frac{\psi_r}{\psi}\right)_r -\frac{\zeta_t}{\zeta} \right]+f_u(t,u)v.
$$
From our assumptions we then
conclude that
$v_t -\Delta_t v\geq f_u v$ in $\tilde \Omega_t.$  The assertion now follows from Lemma \ref{l:estimates}.
\end{proof}

%%%%%%%%%%%%%%%%%%%%%%%%%%%%%%%%%%%%%%%%%%
%%%%%%%%%%%%%%%%%%%%%%%%%%%%%%%%%%%%%%%%%%%SSSSSSSSSSS
\section{T-periodic parabolic problems on Riemannian manifolds}
\subsection{Basic notions from Riemannian geometry.}\label{sec1}
For the reader's convenience we first recall  some notions and results from Riemannian geometry, see e.g. \cite{AMR}. Let $M$ be a Riemannian manifold of dimension $m$ endowed with a metric $\metric=\langle\cdot,\cdot\rangle$. We denote by $p$ an arbitrary point of $M$ and
let $x^1,\ldots, x^m$ be the coordinate functions in  the local chart $U$. Then we have
\begin{equation}
\label{GP1.1}
\metric=g_{ij}\,dx^{i}\otimes dx^{j}
\end{equation}
where $dx^{i}$ denotes the differential of the function $x^{i}$ and
$g_{ij}$ are the (local) components of the metric, defined by
$g_{ij}=\langle\frac{\partial}{\partial x^{i}},
\frac{\partial}{\partial x^{j}}\rangle$.
We will denote by $[g^{ij}]$ the inverse of the matrix $[g_{ij}]$. In the sequel we
shall use the Einstein summation convention over repeated indices.

For any smooth function $u:M\longrightarrow\erre$, the \emph{gradient} of $u$ relative to the metric $g$ of $M$, $\nabla u$, is the vector field dual to the $1$-form $du$, that is
\[
\langle\nabla u, X\rangle = du(X)=X(u)
\]
for all smooth vector fields $X$ on $M$. Note that in local coordinates
we have $\nabla u=u^i\frac{\partial}{\partial x^i}$ with
$$
u^j=g^{ij}\frac{\partial u}{\partial x^i},\qquad\qquad \frac{\partial u}{\partial x^i}=g_{ij}u^j
$$
and
\begin{equation}\label{13}
|\nabla u|^2=\langle \nabla u,\nabla u\rangle= g^{ij}\frac{\partial u}{\partial x^i}\frac{\partial
u}{\partial x^j}.
\end{equation}
The \emph{divergence} of a vector field $X$ on $M$ is given by the trace of $\nabla X$, the covariant derivative of $X$, where $\nabla$ is the (unique) Levi-Civita connection associated to the metric $\metric$. If $X=X^i\frac{\partial}{\partial x^i}$,
it can be expressed in local coordinates as
$$
\operatorname{div}X= \frac{\partial X^i}{\partial
x^i}+X^k\Gamma^i_{ki},
$$
where $\Gamma^k_{ij}$ are the Christoffel symbols
\[
\Gamma^k_{ij} =
\frac{1}{2}g^{kl}\left(\frac{\partial g_{il}}{\partial x^j}+\frac{\partial g_{jl}}{\partial x^i}-\frac{\partial g_{ij}}{\partial x^l}\right).
\]

The Hessian of $u$  is defined as the $2$--tensor $\Hess(u)=\nabla du$ and its components $u_{ij}$ are in local coordinates

$$
u_{ij}=\frac{\partial ^2u}{\partial x^i\partial x^j}
-\Gamma^k_{ij}\frac{\partial u}{\partial x^k}.
$$
Note that
$$
|\operatorname{Hess}(u)|^2= g^{ik}g^{jl}\left(\frac{\partial
^2u}{\partial x^i\partial x^j} -\Gamma^k_{ij}\frac{\partial
u}{\partial x^k}\right)\left(\frac{\partial ^2u}{\partial
x^k\partial x^l} -\Gamma^s_{kl}\frac{\partial u}{\partial
x^s}\right),
$$
and that $\nabla |\nabla u|^2$ is a vector field, which for every smooth vector field $X$ satisfies
\begin{equation}\label{8}
\langle \nabla |\nabla u|^2,X\rangle= 2\operatorname{Hess}(u)(\nabla u,X).
\end{equation}

The
\emph{Laplace--Beltrami operator} of $u$ is the trace of the Hessian, or equivalently the divergence of the gradient, i.e.
\[
\label{GP3.4}
\Delta u =\operatorname{Tr}(\Hess(u))=\operatorname{div}(\nabla u).
\]
In local coordinates it has the form
$$
\Delta u= \mathfrak g^{-1}\frac{\partial}{\partial
x^i}\left(\mathfrak g \, g^{ij} \frac{\partial u}{\partial
x^j}\right),\quad \tx{where} \mathfrak g=\sqrt{\rm{det}(g_{ij})}.
$$

We denote by  $\operatorname{Ric}$ the \emph{Ricci tensor} which  is expressed as
\[
R_{ij}=R_{ji}=\frac{\partial\Gamma^l_{ij}}{\partial x^l}-\frac{\partial\Gamma^{l}_{il}}{\partial x^j}+\Gamma^k_{ij}\Gamma^l_{kl}-\Gamma^k_{il}\Gamma^l_{kj},
\]
Therefore,  if $X=X^i\frac{\partial}{x^i}$, $Y=Y^i\frac{\partial}{x^i}$ are vector fields, we have $\operatorname{Ric}(X,Y)=R_{ij}X^iY^j$.

Next we recall the well--known Bochner--Weitzenb\"ock formula, which will play a crucial role in what follows: for all $u \in C^3(M)$ we have
      \begin{equation}\label{4}
      \frac{1}{2}\Delta \abs{\nabla u}^2 = \abs{\Hess(u)}^2 + \ricc \pa{\nabla u, \nabla u} + \langle\nabla \Delta u, \nabla u\rangle.
      \end{equation}
Moreover for every $u \in C^2(M)$ we have
\begin{equation}\label{21}
      |\nabla|\nabla u|^2|^2\leq4|\operatorname{Hess}(u)|^2|\nabla u|^2,
\end{equation}
see e.g. \cite[formula (3.6)]{BPT}.

If $\Omega\subset M$ is an open set with regular boundary $\partial\Omega$, we will use $\nu$ to denote the outer normal unit vector to $\partial\Omega$  in the tangent space $T_pM$. We shall assume that $\partial\Omega$ is orientable and that the outer normal is well-defined and continuous. Hence for any $p\in\partial\Omega$ there exist a neighborhood $U\subset M$ and a function $\varphi : U \longrightarrow \erre^+$ such that $U \cap \partial\Omega = \varphi^{-1}(0)$ and $\nabla\varphi\neq0$ in $U$.
Then the outward normal unit vector can be computed as $\nu=-\frac{\nabla \varphi}{|\nabla \varphi|}$.

Next we introduce the second fundamental form of $\partial\Omega$ with respect to $M$. For any $p\in\partial \Omega$ let $X,Y$ be in the tangent space of $\partial \Omega$ through $p$, and denote by $\nabla_X\nu$ the
covariant derivative of $\nu$ along $X$. Then
\[
\sff(X,Y)=-\langle \nabla_X\nu,Y\rangle=\langle \nu,\nabla_XY\rangle,
\]
and if $X=X^i\frac{\partial}{\partial x^i}$ and $Y=Y^i\frac{\partial}{\partial x^i}$, in local coordinates it takes the form
\[
\sff(X,Y)=-g_{kl}\left(\frac{\partial \nu^k}{\partial x^i}+\Gamma^k_{ij}\nu^j\right)X^iY^l.
\]

Finally, we recall that for any function $u\in C^2(\overline{\Omega})$ such that $\frac{\partial u}{\partial \nu}=\langle\nabla u,\nu\rangle=0$ on $\partial\Omega$, the vector $\nabla u$ is tangential to $\partial \Omega$,  and
\begin{equation}\label{15}
\frac{1}{2}\frac{\partial}{\partial\nu}|\nabla u|^2=\sff(\nabla u,\nabla u)\qquad\mbox{ on }\partial\Omega,
\end{equation}
see e.g. \cite[Lemma 3.4]{BPT} for the proof.

When we deal with a smooth family of metrics $g(t)=\langle\cdot,\cdot\rangle_t$ with local coefficients $g_{ij}(t)=\langle\frac{\partial}{\partial x^{i}},
\frac{\partial}{\partial x^{j}}\rangle_t$, for $t\in\erre$, we will use $\operatorname{Ric}_t$, $\nabla_t$, $\operatorname{Hess}_t$, $\operatorname{div}_t$, $\Delta_t$,  $|\cdot|_t$, $\sff_t$ and $\nu_t$ to denote the corresponding geometric objects relative to the metric $g(t)$ for $t\in\erre$.
%%%%%%%%%%%%%%%%%%%%%%%%%%%%%%%%%%%%%%%%%%%%%%%%%%%%%%%%%%%%%%%%%
\subsection{Parabolic problem} In this subsection we study stability and instability of $T$-periodic classical solutions of a reaction-diffusion process.
Here $\Omega_t\subseteq M$ is a family of  $T$-periodic submanifolds (with or without boundary) such that $\Omega_t$ is diffeomorphic to $\Omega_0$ for every $t$ under the smooth $T$-periodic map $\Psi(t,\cdot): \Omega_0\rightarrow \Omega_t$. Denote by $g(t)$ the metric induced on $\Omega_t$ by the immersion in $(M,g)$. In the presence of a boundary, $\nu_t$ stands for the unit normal vector field on $\partial\Omega_t$ which belongs to the tangent space of $\Omega_t$.

The process is described by the Neumann boundary value problem
\begin{equation}\label{6}
  \begin{cases}
    \frac{\partial u}{\partial t}-\Delta_tu=f(p,t,u,\nabla_t u) & \mbox{in } Q_T,\\
    \frac{\partial u}{\partial \nu_t}=0 & \mbox{on }\Gamma_T.
  \end{cases}
\end{equation}
If there is no boundary, the Neumann condition in \eqref{6} is dropped.

Here $f:M\times\mathbb{R}\times\mathbb{R}\times TM\longrightarrow\erre$ is a $C^1$, $T$-periodic function and $TM$ is the tangent bundle of $M$.
If at any point on $M$ we identify the tangent bundle $TM$ with its tangent space,  and if we denote by $d_{TM}f(p,t,\xi,\cdot)$ the differential of the map $f(p,t,\xi,\cdot):TM\longrightarrow\erre$, then there exists a continuous, $T$-periodic vector field $X(p,t,u,\nabla_tu)$, such that
\[
d_{TM}f(p,t,u(t,\cdot),\nabla_tu(t,\cdot))V=\langle X(p,t,u(t,\cdot),\nabla_tu(t,\cdot)), V\rangle_t \quad \forall V \in TM.
\]
We now introduce the operator
\begin{equation}\label{3}
  \mathcal{L}:=\frac{\partial}{\partial t}-\Delta_t - \langle X(p,t,u,\nabla_tu), \nabla_t\,\cdot\,\rangle_t-m_0,
\end{equation}
where $m_0=\frac{\partial f}{\partial\xi}(p,t,u,\nabla_tu)$.
The linearization of problem \eqref{6} at $u$ is
\begin{align*}
 \mathcal{L} \phi&=\lambda \phi \quad\tx{in } Q_T, \\
 \nonumber   \frac{\partial \phi}{\partial \nu_t}&=0 \quad\tx{on } \Gamma_T.
 \end{align*}

%%%%%%%%%%%%%%%%%%%%%%%%%%%%%%%%%%%%%%%%%%%%%sssssssss
%%%%%%%%%%%%%%%%%%%%%%%%%%%%%%%%%%
\subsection{Motivation}
 \subsubsection{}Let $u$ denote the density of a substance or of a population which occupies at time $t$ the domain $\Omega_t$
 of $\mathbb{R}^n$. We assume that there is  no flux of the substance across the boundary of the domain, {\sl i.e. for every $t\in\mathbb{R}$}
 \[
\frac{\partial u}{\partial \nu_t}=0\qquad\textrm{ on }\partial\Omega_t\times\{t\},
\]
where $\nu_t$ denotes the outward normal unit vector to $\partial\Omega_t$. If the diffusive flux vector $J$ of the substance obeys Fick's first law, and if in addition there is a source  $F(p,t,u)$,  then in an isotropic media we have
\[
J=-\nabla u +F(p,t,u).
\]

Fick's second law in the presence of an outside force $H(p,t,u,\nabla_0 u)$ implies that for every subdomain $\Omega'\Subset\Omega_t$ the change of the total mass is given by
\begin{align*}
\frac{d}{dt}\int_{\Omega'}u\,dx&= -\int_{\partial\Omega'}\langle J,\nu\rangle\,d\sigma+\int_{\Omega'} H(p,t,u,\nabla u)\,dx,
\end{align*}
where $\nu$ is the outer normal of $\Omega'$ and $d\sigma$ is the surface element of $\partial \Omega'$. By the divergence theorem
\begin{align}\label{Fick2}
\frac{d}{dt}\int_{\Omega'}u\,dx&=\int_{\Omega'}(\Delta u -\operatorname{div} F+H)\:dx.
\end{align}
We now assume that for every $t$, $\Omega_t$  is diffeomorphic to a fixed domain $\Omega_0\subset \mathbb{R}^n$  under the smooth map $\Psi(t,\cdot): \Omega_0\rightarrow \Omega_t$. If in local coordinates the metric tensor of $M$ is $\metric=\delta_{ij}\,dx^{i}\otimes dx^{j}
$, after the change of coordinates $x^i=\psi^i(t,y)$ it becomes
$$
g(t)=\frac{\partial \psi^ i}{\partial y^k}\frac{\partial \psi^ i}{\partial y^s}\,dy^{k}\otimes dy^{s}=:\tilde g_{ks}(y,t)\,dy^{k}\otimes dy^{s} .
$$
The corresponding volume element is $dx=\mathfrak{g}_tdy$.  After the transformation $\psi^{-1}$, \eqref{Fick2} assumes the form
$$
\int_{\psi^{-1}(\Omega')}\left (\frac{\partial u}{\partial t} + \frac{1}{\mathfrak{g}_t}\frac{\partial \mathfrak{g}_t}{\partial t} u\right)\underbrace{\mathfrak{g}_t\,dy}_{d\mu_t}=\int_{\psi^{-1}(\Omega')}(\Delta_t u -\operatorname{div}_t F +H)d\mu_t.
$$
Since this relation holds for arbitrary $\Omega'\Subset\Omega_t$, and consequently $\psi^{-1}(\Omega') \Subset \Omega_0$, we deduce that
\begin{align} \label{diffusion}
\frac{\partial u}{\partial t}=&\Delta_t u -\operatorname{div}_t F +H - \frac{1}{\mathfrak{g}_t}\frac{\partial \mathfrak{g}_t}{\partial t} u \tx{in} \Omega_0\times (0,T),\\
\nonumber &\tilde g^{ks}\frac{\partial u}{\partial y_k}\nu_s =0 \tx{on} \partial \Omega_0\times (0,T),
 \end{align}
where $\nu$ is the outer normal of $\Omega_0$. The divergence of $F$ consists of two parts, namely
$$
\operatorname{div}_t F=\langle\frac{\partial F}{\partial\xi}(p,t,u),\nabla_tu\rangle_t+h(y,t,u),
$$
where $h(\cdot,t,\xi)$ is the divergence of the vector field $F(\cdot,t,\xi)$ on $\Omega_0$ with $t$ and $\xi$ fixed.

\begin{exe}   Suppose that a chemical substance occupies at time $t=0$ the domain is $\Omega_0\subset \mathbb{R}^n$ and as time evolves the domain $\Omega_t\subset  \mathbb{R}^n$. We assume that there is a smooth family of  diffeomorphisms $\psi(t,\cdot):\Omega_0 \to \Omega_t$. The standard Euclidean metric on $\Omega_t$ is after this mapping $\tilde g_{ij}(\cdot,t)=\p_{y^i} \psi(\cdot,t)\cdot\p_{y^j} \psi(\cdot,t)$.

Now suppose that the growth is isotropic in the different directions, i.e. for every component of $x=\psi(t,y)$ we then have
$x^i(y,t)= \rho_i(t) y^i$. Then
\begin{align*}
u_t -\Delta_t u + \frac{\p_t\mathfrak g}{\mathfrak g} u= u_t  -\sum_{i=1}^n \rho_i^{-2} \frac{\p ^2u}{\p x_i ^2} - \sum_{i=1}^n\frac{d \log\rho_i}{dt}u.
\end{align*}
\end{exe}

%%%%%%%%%%%%%%%%%%%%%%%

%%%%%%%%%%%%%%%%%%%%%%%%%%%%%%%%%%%%%%%%
%%%%%%%%%%%%%%%%%%%%%%%%%%%%%%%
\subsection{Main results}

In order to state our main result concerning the instability of the solutions of  \eqref{6}, we need the
 the $2$--tensor $h(t)$ given in local components by
\begin{equation}\label{14}
  h_{ij}(t)=\frac{1}{2}g_{ir}(t)g_{sj}(t)\frac{\partial g^{rs}}{\partial t}(t).
\end{equation}
%%%%%%%%%%%%%%%%%TTTTTTTTTTTT
\begin{theorem}\label{thmA}
Let $f=f(t, u, \nabla_t u):\erre\times\erre\times TM\longrightarrow\erre$ be a $T$-periodic, $C^1$ function  and assume that
$u$ is a $C^3$, $T$-periodic solution of problem \eqref{6} and moreover $\sff_t(V,V)\leq0$ for every vector field $V$ on $\partial\Omega_t$ and every $t\in\erre$.
Assume that for every vector field $X$
\begin{equation}\label{17}
  h(t)(X,X)-\operatorname{Ric}_t(X,X)\leq0.
\end{equation}
If, for some $t\in (0, T)$, $u$ is non constant with respect to $x$, then $u$ is unstable.
\end{theorem}
%%%%%%%%%%%%%%%%%%%%%%%%%tttttttttttttt
The proof is based on the arguments of Hess \cite{He}. The idea is to show that the linearized problem \eqref{4} has a negative eigenvalue. For this purpose we construct a solution
$w$ which satisfies Lemma \ref{l:estimates}.

 It is immediate to see from the proofs that Lemma \ref{l:estimates} and Lemma \ref{l:instability} still apply when we consider a family of submanifolds $\Omega_t\subseteq M$ as above. The operator $A$ is replaced by  \[-\Delta_t- \langle X(p,t,u,\nabla_tu), \nabla_t\,\cdot\,\rangle_t-m_0,\] where $\Delta_t$ is the Laplace--Beltrami operator relative to the metric $g(t)$ and $u$ is a solution of \eqref{6}.

In the proof of the theorem we shall also need the following result.

\begin{proposition}\label{pro1}
 Assume that for every vector field $X$ on $\Omega_t$
\begin{equation}\label{hhh}
  h(t)(X,X)-\operatorname{Ric}_t(X,X)\leq0\qquad\mbox{ on }\Omega_t
\end{equation}
and that $$\sff_t(V,V)\leq0$$ for every vector field $V$ on $\partial\Omega_t$.

Let $f=f(t, u, \nabla_t u):\erre\times\erre\times TM\longrightarrow\erre$ be a $C^1$, $T$-periodic function. Assume also that $u$ is a $C^3$, $T$-periodic solution of problem \eqref{6}.
Then the function $w=|\nabla_t u|_t$ satisfies in the weak sense $\mathcal{L} w\leq0$ on $\bigcup_{t\in\mathbb{R}}\Omega_t\times\{t\}$ and $\frac{\partial w}{\partial\nu_t}\leq0$ on $\bigcup_{t\in\erre}\partial\Omega_t\times\{t\}$.
\end{proposition}

\begin{proof}[Proof of Proposition \ref{pro1}]
The proof is trivial if $u$ depends only on $t\in\erre$, since $w=|\nabla_t u|_t\equiv0$.  From now on we therefore assume that $w\not\equiv0$ on $Q_T$. For any $\epsilon>0$ we consider the function $w_\epsilon:Q_T\longrightarrow\erre$ defined by
\[
w_\epsilon(x,t)=\sqrt{|\nabla_t u(x,t)|_t^2+\epsilon^2}\qquad\mbox{ for }(x,t)\in Q_T.
\]
Then $w\in C^2$ and it is $T$-periodic in $t\in\erre$. Using \eqref{13}, \eqref{8}, \eqref{4} and \eqref{14} we have
\begin{equation}\label{7}
\begin{aligned}
\langle\nabla_tw_\epsilon,\cdot\rangle_t&=\frac{1}{2w_\epsilon}\langle\nabla_t|\nabla_tu|^2_t,\cdot\rangle_t= \frac{1}{w_\epsilon}\operatorname{Hess}_t(u)(\nabla_tu,\cdot)  \\
\Delta_tw_\epsilon&= \frac{1}{w_\epsilon}\left[|\operatorname{Hess}_t(u)|_t^2 + \operatorname{Ric}_t \pa{\nabla_t u, \nabla_t u} + \langle\nabla_t \Delta_t u, \nabla_t u\rangle_t\right]-\frac{1}{4w_\epsilon^3}\big|\nabla_t|\nabla_tu|^2_t\big|^2_t, \\
\frac{\partial}{\partial t}w_\epsilon&=\frac{1}{w_\epsilon}\left[\langle\nabla_t \frac{\partial u}{\partial t}, \nabla_t u\rangle_t+h(t)(\nabla_tu,\nabla_tu)\right],
\end{aligned}
\end{equation}

By \eqref{3} and \eqref{7}, on $Q_T$ we have
\begin{equation}\label{12}
\begin{aligned}
  \mathcal{L}w_\epsilon \,=& \,\,\frac{1}{w_\epsilon}\bigg[\langle\nabla_t \frac{\partial u}{\partial t}, \nabla_t u\rangle_t+h(t)(\nabla_tu,\nabla_tu)-|\operatorname{Hess}_t(u)|_t^2 -\operatorname{Ric}_t \pa{\nabla_t u, \nabla_t u} -\langle\nabla_t \Delta_t u, \nabla_t u\rangle_t\\
  &\,\,-\operatorname{Hess}_t(u)(\nabla_tu,X(t,u,\nabla_tu))\bigg]+\frac{1}{4w_\epsilon^3}\big|\nabla_t|\nabla_tu|_t^2\big|_t^2-m_0w_\epsilon.
\end{aligned}
\end{equation}
Since $u$ is a solution of the differential equation in \eqref{6} we obtain
\begin{equation}\label{10}
\begin{aligned}
\langle\nabla_t \Delta_t u, \nabla_t u\rangle_t&=\langle\nabla_t\frac{\partial u}{\partial t},\nabla_tu\rangle_t-\langle\nabla_t f(t,p,u,\nabla_t u),\nabla_tu\rangle_t\\
&=\langle\nabla_t\frac{\partial u}{\partial t},\nabla_tu\rangle_t-\frac{\partial f}{\partial\xi}(t,u,\nabla_tu)|\nabla_tu|_t^2-\operatorname{Hess}_t(u)(X(t,u,\nabla_tu),\nabla_t u).
\end{aligned}
\end{equation}
Inserting \eqref{10} into \eqref{12}, recalling that $\operatorname{Hess}_t(u)$ is symmetric and that $m_0=\frac{\partial f}{\partial\xi}(t,u,\nabla_tu)$, we obtain
\[
  \mathcal{L}w_\epsilon= \frac{1}{w_\epsilon}\big[h(t)(\nabla_tu,\nabla_tu)-|\operatorname{Hess}_t(u)|_t^2 -\operatorname{Ric}_t \pa{\nabla_t u, \nabla_t u}+m_0|\nabla_tu|_t^2\big]+\frac{1}{4w_\epsilon^3}\big|\nabla_t|\nabla_tu|_t^2\big|_t^2-m_0w_\epsilon.
\]
By \eqref{hhh} and \eqref{21} we have
\[
\frac{1}{w_\epsilon}\big(h(t)(\nabla_tu,\nabla_tu)-\operatorname{Ric}_t \pa{\nabla_t u, \nabla_t u}\big)-\frac{1}{4w_\epsilon^3}\big(4w_\epsilon^2|\operatorname{Hess}_t(u)|_t^2-\big|\nabla_t|\nabla_tu|_t^2\big|_t^2\big)\leq0.
\]
Hence we deduce that
\begin{equation}\label{22}
  \mathcal{L}w_\epsilon\leq-\frac{m_0}{w_\epsilon}\epsilon^2\leq \epsilon\max_{(p,t)\in M\times\erre}|m_0(p,t)|.
\end{equation}
Passing to the limit as $\epsilon$ tends to $0$ we see that $w_\epsilon$ converges to $w=|\nabla_tu|_t$ uniformly and in $W^{1,p}$ for every $p\in[1,\infty)$. Thus passing to the limit in \eqref{22} as $\epsilon$ tends to $0$, we conclude that $\mathcal{L}w\leq0$ on $Q_T$ in the weak sense.

As for the boundary condition, since $\frac{\partial u}{\partial \nu_t}=0$ on $\partial\Omega\times\erre$, by \eqref{15} we have
\[
\frac{\partial w_\epsilon}{\partial \nu_t}=\frac{1}{2w_\epsilon}\frac{\partial }{\partial \nu_t}|\nabla_tu|_t^2=\frac{1}{w_\epsilon}\sff_t(\nabla_t u,\nabla_t u)\qquad\mbox{ on }\bigcup_{t\in\erre}\partial\Omega_t\times\{t\}.
\]
Thus, since by our assumptions we have that $\sff_t(V,V)\leq0$ for every vector field $V$ on $\partial\Omega_t$ and every $t\in\erre$, we deduce that $\frac{\partial w_\epsilon}{\partial\nu_t}\leq0$ on $\partial\Omega_t$ for every $t\in\erre$. Passing to the limit as $\epsilon$ tends to $0$, since $w_\epsilon$ converges to $w=|\nabla_tu|_t$ uniformly and in $W^{1,p}$ for every $p\in[1,\infty)$, we see that $w$ satisfies
\begin{equation*}
  \begin{cases}
    \mathcal{L}w\leq0 & \mbox{in } \bigcup_{t\in\mathbb{R}}\Omega_t\times\{t\}, \\
    \frac{\partial w}{\partial \nu_t}\leq0, & \mbox{on } \bigcup_{t\in\mathbb{R}}\partial\Omega_t\times\{t\}
  \end{cases}
\end{equation*}
in the weak sense.

\end{proof}

\begin{proof}[Proof of Theorem \ref{thmA}]
  The conclusion follows from Proposition \ref{pro1}. Indeed by our assumptions, the function $w=|\nabla_tu|_t$ is nontrivial, $T$-periodic and it satisfies
  \begin{equation*}
  \begin{cases}
    \mathcal{L}w\leq0 & \mbox{in } Q_T, \\
    \frac{\partial w}{\partial \nu_t}\leq0, & \mbox{on } \Gamma_T.
  \end{cases}
\end{equation*}
Then by Lemma \ref{l:estimates} we have that $\lambda_1$, the smallest eigenvalue of problem
   \begin{equation}\label{23}
     \begin{cases}
       \mathcal{L}\phi=\lambda_1\phi&\textrm{in }Q_T,\\
       \frac{\partial \phi}{\partial\nu_t}=0&\textrm{on }\Gamma_T,\\
       \phi(x,t)=\phi(x,t+T)
     \end{cases}
   \end{equation}
is nonpositive. If $\lambda_1<0$ then the solution $u$ is unstable.

  We now show that the case $\lambda_1=0$ cannot occur. By contradiction, assume that $\lambda_1=0$ and let $\phi_1$ be a positive $T$-periodic eigenfunction on $\overline Q_T$ of problem \eqref{23}. Define
  \[
  v(x,t):=\alpha\phi_1(x,t)-w(x,t),\qquad(x,t)\in\overline Q_T,
  \]
   with $\alpha>0$. Since $\phi_1$, $w$ are continuous and $T$-periodic in $t\in\erre$ and since $\overline Q_T$ is compact, $v$ achieves its minimum. We can choose $\alpha>0$ such that $v\geq0$ on $\overline Q_T$ and
  \begin{equation}\label{9}
  \min_{\overline\Omega\times\erre}v=0.
  \end{equation}

Now note that $v$ satisfies $\mathcal{L}v\geq0$ on $Q_T$ and also  $\frac{\partial v}{\partial\nu_t}\geq0$ on $\Gamma_T$, when $\Gamma_T\neq\emptyset$. By the Hopf lemma and by \eqref{9}, if $\Gamma_T\neq\emptyset$ then $v$ cannot achieve its minimum at any point of $\Gamma_T$. Then $v$ achieves its minimum at some point of $Q_T$. Since $v$ is $T$-periodic in $t\in\erre$, by the strong maximum principle
  $v$ must be constant on $\overline Q_T$. Hence by \eqref{9} we have that $v\equiv0$ on $\overline Q_T$, and thus
  \[
  |\nabla_tu|_t=w=\alpha\phi_1
  \]
  is a positive eigenfunction of problem \eqref{23}.

  Since $u$ is continuous, $T$-periodic in $t\in\erre$ and $\overline Q_T$ is compact, $u$ achieves its minimum on $\overline Q_T$ at some point $(x_0,t_0)$ (it would be equivalent to consider a point $(x_1,t_1)$ where $u$ attains its maximum). If $(x_0,t_0)\in\Gamma_T\neq\emptyset$, then the derivative of $u$ in any direction which is tangent to $\partial\Omega_{t_0}$ computed at $(x_0,t_0)$ must vanish, i.e.
  \[
  \langle\nabla_{t_0}u(x_0,t_0),X\rangle_{t_0}=0
  \]
  for every $X\in T_{x_0}\partial\Omega_{t_0}$. Then $\nabla_{t_0}u(x_0,t_0)$ is a scalar multiple of $\nu_{t_0}$; from the boundary condition in \eqref{6} we conclude that
  \begin{equation}\label{5}
  \nabla_{t_0}u(x_0,t_0)=0.
  \end{equation}
  If $(x_0,t_0)\in Q_T$, since $u$ is $T$-periodic we immediately conclude that \eqref{5} holds, since $(x_0,t_0)$ lies in the interior of the domain. Thus we have that
  \[
  w(x_0,t_0)=|\nabla_{t_0}u(x_0,t_0)|_{t_0}=0,
  \]
  which contradicts the positivity of $w=\alpha\phi_1$ we established above.

  Thus the case $\lambda_1=0$ cannot occur, and hence $\lambda_1<0$ and $u$ is unstable.
\end{proof}

The following corollaries, where $(M,g)$ and $\Omega\subseteq M$ are time-independent are consequences of Theorem \ref{thmA}.
\begin{cor}\label{corA}
  Assume that for every vector field $X$ on $\Omega$
\begin{equation}\label{hhhhhh}
  \operatorname{Ric}(X,X)\geq0\qquad\mbox{ on }\Omega.
\end{equation}
Let $f=f(t, u, \nabla u):\erre\times\erre\times TM\longrightarrow\erre$ be a $C^1$ function which is $T$-periodic for $t\in\erre$. Assume also that
$u$ is a $C^3$, $T$-periodic solution of problem
  \[
  \begin{cases}
    u_t-\Delta u=f(t,u,\nabla u) & \mbox{ on } \Omega\times (0,T)\\
    \frac{\partial u}{\partial \nu}=0 & \mbox{on }\partial \Omega \times (0,T).
  \end{cases}
  \]
Moreover let $\sff(V,V)\leq0$ for every vector field $V$ on $\partial\Omega$.
If for some $t\in (0, T)$ the solution $u$ depends on $x\in M$, it is unstable.
\end{cor}

\begin{cor}\label{corB}
 Assume that for every vector field $X$ on $\Omega$
\begin{equation*}
  \operatorname{Ric}(X,X)\geq0\qquad\mbox{ on }\Omega.
\end{equation*}
Let $f=f(u, \nabla u):\erre\times TM\longrightarrow\erre$ be a $C^1$ function. Assume also that
$u$ is a $C^3$ solution of the Neumann problem
  \[
  \begin{cases}
    \Delta u+f(u,\nabla u)=0 & \mbox{ on } \Omega\times (0,T)\\
    \frac{\partial u}{\partial \nu}=0 & \mbox{on }\partial \Omega\times (0,T).
  \end{cases}
  \]
Moreover let $\sff(V,V)\leq0$ for every vector field $V$ on $\partial\Omega$. If $u$ is not constant on $M$, then $u$ is unstable.
\end{cor}

In a bounded domain $\Omega\subset \erre^n$ with $C^1$ boundary $\partial\Omega$ the condition \eqref{hhhhhh} is trivially satisfied, since $\operatorname{Ric}\equiv0$ on $\erre^n$. Moreover the condition that $\sff(V,V)\leq0$ for every vector field $V$ on $\partial\Omega$, requires the domain $\Omega$ to be convex. Thus Theorem \ref{thmA} and Corollary \ref{corA} are in accordance with \cite[Theorem 1]{He}
  and with \cite[Corollary 1]{He}, respectively.

  Moreover, Corollary \ref{corB} extends the results in \cite{Ji} and \cite{BPT} to the case where the nonlinearity $f(u,\nabla u)$ depends also on $\nabla u$, and thus where the first eigenvalue of the corresponding linearized problem does not admit a variational characterization.

Condition \eqref{17} is only needed for $X=\nabla u$, as it clear from the proof. In the case of surfaces of revolution, one knows that a solution depending on $\theta$ is automatically unstable, see Lemma \ref{lemma33}, hence condition \eqref{17} is only needed for radial solutions, i.e. in the radial direction. Hence in this case \eqref{17}  reduces to $\operatorname{Ric}\geq0$. By Remark \ref{rem33}, Theorem \ref{thmB} is sharper than Theorem \ref{thmA} for surfaces of revolution.

%%%%%%%%%%%%%%%%%%%%%%%%%%%%%%%%%%%%%%%%%%%%%%%%%
%%%%%%%%%%%%%%%%%%%%%%%%%%%%%%%%%%%%%%%%%%%%%%%%%
\bibliographystyle{plain}

\end{document}